\documentclass[twoside]{amsart}

\numberwithin{equation}{section}
\usepackage{amsmath,amssymb,amsfonts,amsthm,latexsym}
\usepackage{graphicx,color,enumerate}
\usepackage[all]{xy}

\newtheorem{theorem}{Theorem}[section]
\newtheorem{lemma}[theorem]{Lemma}
\newtheorem{proposition}[theorem]{Proposition}
\newtheorem{corollary}[theorem]{Corollary}
\newtheorem{conjecture}[theorem]{Conjecture}

\theoremstyle{definition}
\newtheorem{definition}[theorem]{Definition}

\theoremstyle{remark}
\newtheorem{remark}[theorem]{Remark}

\numberwithin{equation}{section}

\newcommand{\C}{\mathbb{C}}
\newcommand{\Sp}{\mathbb{S}}

\newcommand{\R}{\mathbb{R}}
\newcommand{\T}{\mathbb{T}}

\newcommand{\Z}{\mathbb{Z}}
\newcommand{\N}{\mathbb{N}}

\newcommand{\U}{\operatorname{U}}
\newcommand{\id}{\operatorname{id}}

\newcommand{\diag}{\operatorname{diag}}
\newcommand{\co}{\colon\thinspace}

\newcommand{\bs}{\boldsymbol}
\newcommand{\cc} [1] {\overline {{#1}}}

\newcommand{\Hilb}{\operatorname{Hilb}}

\newcommand{\calC}{\mathcal{C}}

\begin{document}

\title{On compositions with $x^2/(1-x)$}

\author{Hans-Christian Herbig}
\address{Charles University in Prague, Faculty of Mathematics and Physics,
Sokolovsk\'{a} 83, 186 75 Praha 8, Czech Republic}
\email{herbig@imf.au.dk}

\author{Daniel Herden}
\address{University of Duisburg-Essen, Department of Mathematics,
Campus Essen, 45117 Essen, Germany}
\email{daniel.herden@uni-due.de}

\author{Christopher Seaton}
\address{Department of Mathematics and Computer Science,
Rhodes College, 2000 N. Parkway, Memphis, TN 38112}
\email{seatonc@rhodes.edu}

\thanks{The first and second author were supported by the grant GA CR P201/12/G028.
The third author was supported by a Rhodes College Faculty Development Grant as well as
the E.C. Ellett Professorship in Mathematics.}

\begin{abstract}
In the past, empirical evidence has been presented that Hilbert series of symplectic quotients of unitary representations obey a certain universal  system of infinitely many constraints. Formal series with
this property have been called \emph{symplectic}. Here we show that a formal power series is symplectic if and only if it is a formal composite with the formal power series $x^2/(1-x)$. Hence the set of symplectic power series forms a subalgebra of the algebra of formal power series. The subalgebra property is translated into an identity for the coefficients of the even Euler polynomials, which can be interpreted as a cubic identity for the Bernoulli numbers. Furthermore we show that a rational power series is symplectic if and only if it is invariant under the idempotent M\"{o}bius transformation
$x\mapsto x/(x-1)$. It follows that the Hilbert series of a graded Cohen-Macaulay algebra $A$ is symplectic if and only if  $A$ is Gorenstein with its a-invariant and its Krull dimension adding up to zero. It is shown that this is the case for algebras of regular functions on symplectic quotients of unitary representations of tori.
\end{abstract}

\maketitle

\tableofcontents


\section{Introduction}
\label{sec:Intro}


Let $G\to \U(V)$ be a unitary representation of a compact Lie group $G$ on a
hermitian vector space $(V,\langle\:,\:\rangle)$. Here, $V$ is viewed as a symplectic manifold or real variety. The $\R$-algebra of smooth functions on $V$ is denoted $\calC^\infty(V)$,
and its subalgebra of real regular functions is denoted $\R[V]$. Note that $\R[V]$ is actually a Poisson subalgebra of $\calC^\infty(V)$.
The symplectic form on $V$ is given by the imaginary part of the scalar product $\langle\:,\:\rangle$, and the $G$-action on $V$ is Hamiltonian with moment map
\[
    J:V\to \mathfrak  g^*, \quad J_\xi(v):=(J(v),\xi):=\frac{\sqrt{-1}}{2}\langle v,\xi.v\rangle.
\]
Here, $\xi.v:=d/dt_{t=0}\left(\exp(-t\xi).v\right)$ denotes the infinitesimal action of $\xi\in \mathfrak g$ on $v\in V$ and $(\:,\:)$ stands for the dual pairing between the dual space $\mathfrak g^*$ and the Lie algebra $\mathfrak g$ of $G$.

Let us denote by $Z:=J^{-1}(0)$ the zero fibre of the moment map. If $G$ is finite, then $J=0$ by convention and $Z=V$. Since $J$ is $G$-equivariant, we can consider the space $M_0:=Z/G$ of $G$-orbits in $Z$, the so-called \emph{symplectic quotient}. It is a stratified symplectic space and can be viewed in a natural way as a semialgebraic set (for more information the reader may consult \cite{SjamaarLerman}).
In order to define the smooth structure on $M_0$, one  introduces the vanishing ideal $I_Z$ of $Z$ inside $\calC^\infty(V)$. Then the algebra of smooth functions on $M_0$ is given by $\calC^\infty(M_0):=\calC^\infty(V)^G/(I_Z\cap \calC^\infty(V)^G)$.  Note that
$\calC^\infty(M_0)$ carries a canonical Poisson bracket.
The $\N$-graded $\R$-algebra of regular functions $\R[M_0]:=\R[V]^G/(I_Z\cap \R[V]^G)$ is a Poisson subalgebra of $\calC^\infty(M_0)$.

In this paper, we are concerned with the Hilbert series of the $\N$-graded algebra
$\R[M_0]$. This is the generating function counting the dimensions $\dim_{\R}(\R[M_0]_i)$ of the spaces of regular functions of degree $i\in\N$:
\[
    \Hilb_{\R[M_0]}(t):=\sum_{i\ge 0}\dim_{\R}(\R[M_0]_i)\:t^i\in \N[\![t]\!]\subset\C[\![t]\!].
\]
The Poisson brackets will play no role in the considerations to follow.

The main motivation for our investigation is Conjecture \ref{HigherRelations} below, that has been formulated in \cite{HerbigSeaton}. We recall the following definition from  \cite{HerbigSeaton}.

\begin{definition} \label{def:symplectic} For a formal power series $\varphi(x)=\sum_{i \ge 0} \gamma_i\: x^i\in \C[\![x]\!]$
and $m\ge 1$ we introduce the linear constraint
\begin{equation}
\tag{$\operatorname S_m$}\label{Sm}
    \sum\limits_{k=0}^{m-1} (-1)^k {m-1 \choose k} \gamma_{m+k} = 0.
\end{equation}
We say that $\varphi(x)$ is \emph{symplectic} if condition (\ref{Sm}) holds for each $m\ge 1$.
A meromorphic function $\psi(t)$ in the variable $t$ is said to be \emph{symplectic at $a\in \C$ of pole order $d\in \Z$} if the formal power series $x^d\psi(a-x)\in \C[\![x]\!]$ is symplectic.  Here we assume that the order of the pole of $\psi(t)$ at $t=a$ is $\le d$.
\end{definition}

The reader is invited to check that in a symplectic power series $\varphi(x)=\sum_{i \ge 0} \gamma_i\: x^i\in \C[\![x]\!]$ the odd coefficients $\gamma_1,\gamma_3,\gamma_5,\dots$
are uniquely determined by the even ones $\gamma_0,\gamma_2,\gamma_4,\dots$. Moreover, for each choice of the even coefficients $\gamma_0,\gamma_2,\gamma_4,\dots$ there is a uniquely determined symplectic power series $\varphi(x)=\sum_{i \ge 0} \gamma_i\: x^i$.

The curious sign convention (we expand in powers of $(a-x)$ instead of $(x-a)$) appears to be more natural, because in this way our typical examples render non-negative coefficients.  When we say a meromorphic function is symplectic at $x = a$ of order $d=0$,
we mean that it is analytic at $x = a$ and symplectic as a series expanded in $(a - x)$.  Note that we only use this sign convention
for a formal power series in the context of Lemma \ref{lem:RatCompos}.

\begin{conjecture}[\cite{HerbigSeaton}] \label{HigherRelations} Let $G\to U(V)$ be a unitary representation of a compact Lie group $G$ and let $\R[M_0]$ be the graded $\R$-algebra of regular functions on the corresponding symplectic quotient $M_0$.  Then $\Hilb_{\R[M_0]}(t)$ is symplectic at $t=1$ of order $d=\dim_\R(M_0)$.
\end{conjecture}

There is an analogue of this conjecture for \emph{cotangent lifted representations} of reductive complex Lie groups. Certainly, over the complex numbers, there exist also  symplectic quotients that arise from non-cotangent lifted representations whose Hilbert series are symplectic.
For instance, the invariant ring of any unimodular representation of a finite group has  a symplectic Hilbert series;
for more details, see Section \ref{sec:Calculations}.  To name a specific example,
for $n\ge 2$ the action of the binary dihedral group $\mathbb D_n\subset  \operatorname{SL}_2(\C)$ on
$\C^2$  cannot be cotangent lifted as there are no quadratic invariants.

Our aim is to give a simple proof of the following statement.

\begin{theorem}\label{TheoremTorus}
Conjecture \ref{HigherRelations} holds if $G$ is a torus.
\end{theorem}

The crucial insight that helps us to achieve our goal is the following reformulation of
what it means for a generating function to be symplectic.

\begin{proposition} \label{FormalComposite} A formal power series $\varphi(x)$ is symplectic if and only if it is a formal composite with $x^2/(1-x)$, i.e., if there exists a formal power series $\rho(y)\in  \C[\![y]\!]$ such that $\varphi(x)=\rho(x^2/(1-x))$.
\end{proposition}

As a corollary, we obtain the following.

\begin{corollary}
\label{cor:Product}
The space of symplectic power series forms a subalgebra of $\C[\![x]\!]$. A meromorphic function
$\psi(t)$ is symplectic of order $d$ at $a\in\C$ if and only if there exists a formal power series  $\rho(y)\in  \C[\![y]\!]$ such that the Laurent expansion of  $\psi(t)$ at $t=a$ is
\[
    \frac{1}{(a-t)^d}\:\rho\!\left(\frac{(a-t)^2}{1-a+t}\right).
\]
If $\psi_1(t)$ is symplectic at $a\in\C$
of order $d_1$ and $\psi_2(t)$ is symplectic at $a\in\C$ of order $d_2$, then the product
$\psi_1(t)\psi_2(t)$ is symplectic at $a\in\C$ of order $d_1 + d_2$.
\end{corollary}

It is tempting to think of $x^2/(1-x)$ as some sort of fundamental (rational or formal) invariant of a group action. In fact, the requisite transformation is provided by the order two M\"obius transformation $x\mapsto x/(x-1)$.

\begin{theorem} \label{TheoremMoebius}
A formal power series $\varphi(x)$ is symplectic if and only if it is invariant
under the substitution $x\mapsto x/(x-1)$. If  $\varphi(x)$ is rational, then the following
statements are equivalent:
\begin{enumerate}
\item $\varphi(x)$ is symplectic,
\item there exists a rational function $\rho(y)$ such that $\varphi(x)=\rho(x^2/(1-x))$,
\item  $\varphi(x)=\varphi(x/(x-1))$.
\end{enumerate}
\end{theorem}

\begin{corollary}
\label{CorRatSympOrderD}
A rational function $\psi(t)$ is symplectic of order $d$ at $t=a$ if and only if
\begin{eqnarray}\label{FuncEq t=a}
\psi\!\left(\frac{a^2-2a+(1-a)t}{a-1-t}\right)=(a-1-t)^d\:\psi(t).
\end{eqnarray}
\end{corollary}

This type of functional equation one encounters in the theory of \emph{Gorenstein algebras} (cf. \cite[Section 4.4]{BrunsHerzog}). Namely, by a theorem of Richard P. Stanley \cite{Stanley}, an $\N$-graded Cohen-Macaulay algebra $R=\oplus_{i\ge 0}R_i$ is Gorenstein if and only if its Hilbert series
$\Hilb_R(t)=\sum_{i\ge 0} \dim(R_i)\:t^i$ fulfills
\begin{eqnarray}
\Hilb_R(t^{-1})=(-1)^d t^{-a(R)} \Hilb_R(t),
\end{eqnarray}
where $d=\dim(R)$ and $a(R)$ is the so-called a-invariant. By comparison with \eqref{FuncEq t=a} for $a=1$ we finally obtain the following result.

\begin{corollary} The Hilbert series $\Hilb_R(t)$ of a graded Cohen-Macaulay algebra
 $R$ is symplectic of order $d=\dim R$ if and only if $R$ is Gorenstein with a-invariant $a(R)=-d$.
\end{corollary}

\begin{remark}
In particular, this implies that if the graded ring $R=\oplus_{i\ge 0} R_i$  is Gorenstein of Krull dimension $d$ and in the Laurent expansion
\begin{eqnarray}\label{eq:LaurentExpansion}
\Hilb_{R}(t)=\sum_{i\ge 0}{\gamma_i\over (1-t)^{d-i}},
\end{eqnarray}
the coefficient $\gamma_1=0$, then $\Hilb_{R}(t)$ is symplectic of order $d$. Here we make use of the fact \cite[Equation (3.32)]{PopovVinberg} that $-2\gamma_1/\gamma_0=a(R)+d$.
\end{remark}

Let us give an outline of the paper. In Section \ref{sec:FormalComposite} we prove Proposition \ref{FormalComposite} and, as a side remark, discuss relations to the sequence of Genocchi numbers. In Section \ref{sec:MoebiusInvariance} we use Proposition \ref{FormalComposite} to show Theorem \ref{TheoremMoebius}. The latter is used in Section
\ref{sec:Torus} to give a proof of our main result, Theorem \ref{TheoremTorus}, that is based on Molien's formula and the fact that a moment map of a faithful torus representation forms a regular sequence in the ring of invariants \cite{HerbigIyengarPflaum,FarHerSea}. In Section  \ref{sec:Euler} we deduce from Corollary \ref{cor:Product} an identity for the coefficients of the even Euler polynomials. In Section \ref{sec:Calculations} we illustrate our results by discussing specific examples.

\subsection*{Acknowledgements}

We would like to thank Srikanth Iyengar for suggesting that condition (\ref{Sm}) might be fulfilled termwise in the Molien formula for a finite unitary group.  The third author would like to thank Eric Gottlieb for helpful conversations.


\section{Proof of Proposition \ref{FormalComposite}}
\label{sec:FormalComposite}

As, for each  $i\ge 1$, the alternating sum over the $i$th row of the Pascal triangle is zero, the power series $x^2/(1-x)$ is symplectic. Based on this observation we are able to find more examples.

\begin{lemma} For each $n\ge 0$ the formal power series $(x^2/(1-x))^n$ is symplectic.
\end{lemma}
\begin{proof} First let us observe that for a formal power series $\varphi(x)$ we have
\begin{eqnarray*}
\eqref{Sm}\quad\Longleftrightarrow \quad\left.{d^{2m-1}\over d^{2m-1}x}\right|_{x=0}\Big((1-x)^{m-1}\varphi(x)\Big)=0.
\end{eqnarray*}
Let us introduce the shorthand notation $f_{n,m}(x):=(1-x)^{m-1}\left(x^2/(1-x)\right)^n$. The rational function $f_{n,m}(x)$ is regular at $x=0$ and vanishes there to the order $2n$. So if $m\le n$, then
\[f_{n,m}^{(2m-1)}(0)=0.\]
On the other hand, if $m> n$, then $f_{n,m}(x)$ is a polynomial of degree $n+m-1<2m-1$ and hence the $(2m-1)$-fold derivative of $f_{n,m}$ vanishes identically.
\end{proof}

It will be convenient to introduce some terminology.

\begin{definition}
\label{def:SympBasis}
By a \emph{symplectic basis} we mean  a
sequence $(\varphi_n(x))_{n\in \N}$ of symplectic power series $\varphi_n(x)\in \C[\![x]\!]$
such that each $\varphi_n(x)\in\mathfrak m^{2n}$ and its class in $\mathfrak m^{2n}/\mathfrak m^{2n+1}$ is nonzero. Here $\mathfrak m$ denotes the maximal ideal $x\,\C[\![x]\!]$
of the complete local ring $\C[\![x]\!]$.
\end{definition}

\begin{lemma}\label{lem:m-adic}
Let $(\varphi_n(x))_{n\in \N}$ be a symplectic basis. Then for each symplectic power series
$\varphi(x)$ there exists a unique sequence $(a_n)_{n\in\N}$ of numbers such that
for each $k\ge 0$
\begin{eqnarray}\label{eq:finiteApprox}
\varphi(x)-\sum_{i=0}^k a_i\, \varphi_i(x)\in \mathfrak m^{2k+2}.
\end{eqnarray}
It follows that $\varphi(x)=\sum_{i\ge 0} a_i\, \varphi_i(x)$, where the sum converges in
the $\mathfrak m$-adic topology of $\C[\![x]\!]$.
\end{lemma}
\begin{proof} We start with a preparatory observation. Suppose that $k\ge 0$ and $f(x)=\sum_{i\ge 0}\alpha_i\,x^i$ is symplectic and in $\mathfrak m^{2k+1}$, i.e., $\alpha_0=\alpha_1=\dots=\alpha_{2k}=0$.  Then $(\operatorname S_{k+1})$ implies that $\alpha_{2k+1}=0$ as well, that is $f(x)\in\mathfrak m^{2k+2}$.

Assume now for induction that
\[\varphi(x)-\sum_{i=0}^{k-1} a_i\, \varphi_i(x)\in \mathfrak m^{2k}.\]
It follows that there is a unique number $a_k$ such that $\varphi(x)-\sum_{i=0}^{k} a_i\, \varphi_i(x)\in \mathfrak m^{2k+1}$. Since the latter series is symplectic, the above argument tells us that it is in fact in $\mathfrak m^{2k+2}$.
\end{proof}

As a consequence, with the choice of the symplectic basis
\begin{eqnarray}
\label{eq:rationalSymplectic}
\left(\left({x^2\over 1-x}\right)^n\right)_{n\in \N}
\end{eqnarray}
we can write each symplectic series $\varphi(x)$ as a formal composite $\rho(x^2/(1-x))$, where $\rho(y)=\sum_{i\ge 0} a_i\,y^i\in \C[\![y]\!]$. This proves Proposition \ref{FormalComposite}.

\begin{remark}\label{rmk:Genocchi}
There are of course plenty of other symplectic bases. In fact, any symplectic power series
$\varphi_1(x)$ that is in $\mathfrak m^2$ and whose class in $\mathfrak m^2/\mathfrak m^3$
does not vanish generates a symplectic basis $\left((\varphi_1(x))^n\right)_{n\in\N}$.
A choice different from $x^2/(1-x)$ is provided by the sequence of \emph{Genocchi numbers}.
The sequence of Genocchi numbers $(G_n)_{n\in\N}$ (cf. entry A036968 in the online encyclopedia \cite{OEIS}) is defined by the exponential generating function
\begin{eqnarray*}
\frac{2z}{e^z+1}&=&\sum_{n\ge 0}G_n\frac{z^n}{n!}\\
&=&z-\frac{z^2}{2!}+\frac{z^4}{4!}-\frac{3z^6}{6!}+\frac{17z^8}{8!}-\frac{155z^{10}}{10!}+\frac{2073z^{12}}{12!}-\dots\quad \in \C[\![z]\!].
\end{eqnarray*}
Setting $\operatorname{Gen}(x):=\sum_{i\ge 0}G_{n+1}x^n$, it follows from \cite{GesselUmbral} (see also Section \ref{sec:Euler}) that
\begin{eqnarray}\label{eq:phi1}
\varphi_1(x):=x^2\operatorname{Gen}(-x)
\end{eqnarray} is symplectic, and hence generates a symplectic basis as described above.
As $G_n = O(n!/\pi^n)$, $\operatorname{Gen}(x)$ as well as $\varphi_1(x)$ cannot be rational.
Also note that the only even monomial occurring in the expansion of $\varphi_1(x):=x^2\operatorname{Gen}(-x)$ is $x^2$.
\end{remark}

\section{Proof of Theorem \ref{TheoremMoebius}}
\label{sec:MoebiusInvariance}

First let us prove that a formal power series $\varphi(x)=\sum_{i\ge 0}\gamma_i \,x^i$ is symplectic if and only if
\begin{eqnarray}\label{eq:invariance}
\varphi(x)=\varphi(x/(x-1)).
\end{eqnarray}
The implication $\Longrightarrow$
is a consequence of Proposition
\ref{FormalComposite}. Conversely, let us assume that $\varphi(x)$ fulfills  Equation \eqref{eq:invariance}. Using the identity
\[\left({x\over x-1}\right)^n=\sum_{i\ge 0}(-1)^n{n+i-1\choose n-1}x^{n+i},\]
for $n\ge 1$, we see that  Equation \eqref{eq:invariance} is tantamount to
\begin{eqnarray}\label{eq:recursion}
\gamma_m=\sum_{i=1}^m(-1)^n{m-1\choose n-1}\gamma_n
\end{eqnarray}
for all $m\ge 1$. Without loss of generality we can assume that $\gamma_{2n}=0$ for all $n\ge 0$. This can be achieved by subtracting a suitable symplectic power series. With this assumption it follows recursively from \eqref{eq:recursion} that $\gamma_n=0$ for all $n\ge 0$. Since $\varphi(x)=0$ is symplectic, this shows implication $\Longleftarrow$.

This establishes the first claim of Theorem \ref{TheoremMoebius}.  The rest of the statement will follow from the following.

\begin{lemma}
\label{lem:RatCompos}
Let $\varphi(x)=P(x)/Q(x)$ be a rational symplectic function.
Then there exists a rational function $\rho(y)$ such that
$\varphi(x)=\rho(x^2/(1-x))$.
\end{lemma}
\begin{proof}
Assume that $\varphi$ is nonzero, and then we may express
\begin{equation}
\label{eq:fDef}
    \varphi(x) = C x^k (x - 1)^\ell (x - 2)^m \prod\limits_{i=1}^r (x - \lambda_i)^{n_i}
\end{equation}
where $C \in \C^\times$, each $\lambda_i \in \C\smallsetminus\{0,1,2\}$, $r$ is a nonnegative
integer, and  $k$, $\ell$, $m$, and $n_i$ for $i = 1,\ldots, r$ are integers.  Let
$q = \deg(Q(x))-\deg(P(x))$, and then we have
\begin{equation}
\label{eq:dqDef}
    q = -k - \ell - m - \sum\limits_{i=1}^r n_i.
\end{equation}
By a simple computation,
\begin{equation}
\label{eq:fSub}
    \varphi\left( \frac{x}{x-1} \right)
    =
    C (-1)^m x^k (x - 1)^{q} (x - 2)^m
    \prod\limits_{i=1}^r (1 - \lambda_i)^{n_i}
    \left( x - \frac{\lambda_i}{\lambda_i - 1}\right)^{n_i}.
\end{equation}
We have that $\varphi(x)$ is symplectic by hypothesis so that by the substitution theorem
\cite[Theorem 9.25]{ApostolAnalysis}, Equation \eqref{eq:invariance} holds for
$\varphi(x)$.  Hence, a comparison of Equations \eqref{eq:fDef} and \eqref{eq:fSub} yields
\begin{align}
    \label{eq:fSubCons1}
    (-1)^m \prod\limits_{i=1}^r (1 - \lambda_i)^{n_i}
    &=  1,
    \\
    \label{eq:fSubCons2}
    q      &=      \ell,
    \quad\quad\mbox{and}\\
    \label{eq:fSubCons3}
    \prod\limits_{i=1}^r \left( x - \frac{\lambda_i}{\lambda_i - 1}\right)^{n_i}
    &=
    \prod\limits_{i=1}^r (x - \lambda_i)^{n_i}.
\end{align}
From Equation \eqref{eq:fSubCons3}, we have that for each factor
$x - \lambda_i$ in $\prod_{i=1}^r (x - \lambda_i)^{n_i}$, a factor
$x - \lambda_i/(\lambda_i - 1)$ must also appear.  Hence we may rewrite
\begin{equation}
\label{eq:fsubMuFactors}
    \prod\limits_{i=1}^r (x - \lambda_i)^{n_i}
    =
    \prod\limits_{j=1}^{r^\prime} (x - \mu_j)^{n_j^\prime}
    \left( x - \frac{\mu_j}{\mu_j - 1}\right)^{n_j^\prime}
\end{equation}
for a nonnegative integer $r^\prime$, nonnegative integers $n_j^\prime$
and $\mu_j \in \C\smallsetminus \{0,1,2\}$ for $j=1,\ldots, r^\prime$.
Combining Equations \eqref{eq:fSubCons1} and \eqref{eq:fsubMuFactors} and observing
that for each $j$, $(1 - \mu_j)(1 - \mu_j/(\mu_j - 1)) = 1$, we obtain
\[
    1
    =
    (-1)^m \prod\limits_{j=1}^{r^\prime} (1 - \mu_j)^{n_j^\prime}
    \left( 1 - \frac{\mu_j}{\mu_j - 1}\right)^{n_j^\prime}
    =
    (-1)^m
\]
so that $m = 2m^\prime$ for some $m^\prime \in \Z$.  Similarly, Equations
\eqref{eq:dqDef} and \eqref{eq:fSubCons2} can now be used to express
\[
    k = - 2\ell - 2m^\prime - \sum\limits_{j=1}^{r^\prime} 2n_j^\prime
\]
so that $k = 2k^\prime$ for some $k^\prime \in \Z$, and then
we have
\[
    \ell =  -k^\prime - m^\prime - \sum\limits_{j=1}^{r^\prime} n_j^\prime.
\]
Substituting \eqref{eq:fsubMuFactors} into \eqref{eq:fDef} and applying the
above observations yields
\begin{align*}
    \varphi(x)
    &=
    C x^{2k^\prime} (x - 1)^{-k^\prime - m^\prime - \sum_{j=1}^{r^\prime} n_j^\prime}
    (x - 2)^{2m^\prime}\prod\limits_{j=1}^{r^\prime} (x - \mu_j)^{n_j^\prime}
    \left( x - \frac{\mu_j}{\mu_j - 1}\right)^{n_j^\prime}
    \\&=
    C \left(- \frac{x^2}{1-x}\right)^{k^\prime}
    \left(-\frac{x^2}{1-x} - 4\right)^{m^\prime}
    \prod\limits_{j=1}^{r^\prime} \left( - \frac{x^2}{1 - x} - \frac{\mu_j^2}{\mu_j - 1}\right)^{n_j^\prime},
\end{align*}
a rational function of $x^2/(1-x)$, completing the proof.
\end{proof}


\section{Proof of Theorem \ref{TheoremTorus}}
\label{sec:Torus}

In this section, we let $G = \T^\ell = (\Sp^1)^\ell$, let $V$ be a unitary representation
of $G$ with $\dim_\C V = n$, and let $M_0$ denote the corresponding symplectic quotient.
We choose a (complex) basis for $V$ with respect to which the $G$-action
is diagonal, and then the action of $G$ is described by a \emph{weight matrix}
$A \in \Z^{\ell\times n}$. Specifically, we let $\bs{z} := (z_1, \ldots, z_\ell) \in G$ with
each $z_i \in \Sp^1$ and introduce the notation
$\bs{z}^{\bs{a}_j} := z_1^{a_{1,j}} z_2^{a_{2,j}} \cdots z_\ell^{a_{\ell,j}}$ for each
$j=1,\ldots,n$.  Then the action of $\bs{z}$ on $V$ as a unitary transformation is given with respect to this basis by
\[
    \bs{z}\mapsto
    \diag(\bs{z}^{\bs{a}_1}, \ldots, \bs{z}^{\bs{a}_n}).
\]
Concatenating our basis for $V$ with its complex conjugate to produce a real basis for $V$, the action of $\bs{z}$ on $V$ as real linear transformations is given by
\[
    \bs{z}\mapsto
    \diag(\bs{z}^{\bs{a}_1}, \ldots, \bs{z}^{\bs{a}_n},
    \bs{z}^{-\bs{a}_1}, \ldots, \bs{z}^{-\bs{a}_n}).
\]

Let $J \co V \to \mathfrak{g}^\ast$ denote the homogeneous quadratic moment map, let
$Z:= J^{-1}(0)$, and let $M_0 := Z/G$ denote the symplectic quotient; see
Section \ref{sec:Intro}.  As $G$ is abelian, the components of $J$ are elements of
$\R[V]^G$.  We may assume without loss of generality that $0$ is in the convex
hull of the columns of $A$ in $\R^\ell$ and the rank of $A$ is $\ell$;
see \cite[Section 2]{HerbigIyengarPflaum} or \cite[Section 3]{FarHerSea}.

Using Molien's formula, see \cite[Section 4.6.1]{DerskenKemperBook},
the Hilbert series of the invariant ring $\R[V]^G$ is given by
\[
    \Hilb_{\R[V]^G}(t)
    =
    \frac{1}{(2\pi i)^\ell} \int\limits_{\bs{z}\in\T^\ell}
    \frac{dz_1 dz_2 \cdots dz_\ell}
    {(\prod_{j=1}^n z_j ) \prod_{j=1}^n (1 - t \bs{z}^{\bs{a}_j}) (1 - t \bs{z}^{-\bs{a}_j})}.
\]
Then by \cite[Proposition 2.1]{HerbigSeaton}, the Hilbert series of the real regular functions
on the symplectic quotient $M_0$ is given by
\[
    \Hilb_{\R[M_0]}(t) =
    \frac{1}{(2\pi i)^\ell} \int\limits_{\bs{z}\in\T^\ell}
    \frac{(1 - t^2)^\ell dz_1 dz_2 \cdots dz_\ell}
    {(\prod_{j=1}^n z_j ) \prod_{j=1}^n (1 - t \bs{z}^{\bs{a}_j}) (1 - t \bs{z}^{-\bs{a}_j})}.
\]
Define the function
\[
    h(\bs{z},t)
    =
    \frac{(1 - t^2)^\ell}
    {(\prod_{j=1}^n z_j ) \prod_{j=1}^n (1 - t \bs{z}^{\bs{a}_j}) (1 - t \bs{z}^{-\bs{a}_j})},
\]
and then we have
\begin{align*}
    h(\bs{z},t^{-1})
    &=
    \frac{(1 - t^{-2})^\ell}
    {(\prod_{j=1}^n z_j ) \prod_{j=1}^n (1 - t^{-1} \bs{z}^{\bs{a}_j}) (1 - t^{-1} \bs{z}^{-\bs{a}_j})}
    \\&=
    \frac{t^{2(n-\ell)} (t^2 - 1)^\ell}
    {(\prod_{j=1}^n z_j ) \prod_{j=1}^n (1 - t \bs{z}^{\bs{a}_j}) (1 - t \bs{z}^{-\bs{a}_j})}
    \\&=
    (-1)^\ell t^{2(n-\ell)} h(\bs{z},t).
\end{align*}

Fix $t \in \C$ with $|t| < 1$, and then
\begin{align*}
    \Hilb_{\R[M_0]}(t^{-1})
    &=
    \frac{1}{(2\pi i)^\ell} \int\limits_{\bs{z}\in\T^\ell} h(\bs{z},t^{-1}) dz_1 dz_2 \cdots dz_\ell
    \\&=
    \frac{t^{2(n-\ell)}}{(2\pi i)^\ell} \int\limits_{\bs{z}\in\T^\ell} (-1)^\ell h(\bs{z},t) dz_1 dz_2 \cdots dz_\ell .
\end{align*}
Choose an $i$ and fix arbitrary values $z_k \in \Sp^1$ for $k\neq i$.  Dividing the numerator
and denominator by $z_i^{a_{i,j}}$ for each $a_{i,j} < 0$ to express $h(\bs{z},t)$ in terms of positive
powers of $z_i$, and using the fact that each row of $A$ contains at least one nonzero entry, it is easy to see that
\[
    \operatorname{Res}_{z_i=\infty} h(\bs{z},t)
    =
    -\operatorname{Res}_{z_i=0} \frac{1}{z_i^2} h(z_1, \ldots, 1/z_i, \ldots, z_n,t) = 0.
\]
A computation demonstrates that the transformation $t \mapsto t^{-1}$ induces a bijection between
the poles in $z_i$ inside the unit disk with those outside the unit disk.  Then
considering each $\Sp^1$-factor of $\T^\ell$ as a negatively-oriented curve about the
point at infinity, introducing a factor of $(-1)^\ell$, we have
\[
   \Hilb_{\R[M_0]}(t^{-1})
    =
    t^{2(n-\ell)} \Hilb_{\R[M_0]}(t).
\]
Then Theorem \ref{TheoremTorus} follows from Corollary \ref{CorRatSympOrderD}
and the fact that $\R[M_0]$ has dimension $d = 2(n-\ell)$.


\section{Applications to even Euler polynomials and Bernoulli Numbers}
\label{sec:Euler}

Our aim is to derive from the fact that the space of symplectic power series forms
a subalgebra (cf. Corollary \ref{cor:Product}) a certain combinatorial identity, Equation \eqref{eq:Eulerrelations}, that might be of independent interest. We recall that the  Euler polynomials $E_n(x)$, $n=0,1,2,\dots$, are defined by the expansion
\begin{equation}\label{Euler}
\frac{2e^{xt}}{e^t+1}=\sum_{n\ge 0} E_n(x)\frac{t^n}{n!}.
\end{equation}
We introduce numbers ${n \brack i}$ via
\[x(x^{2n}-E_{2n}(x))=:\sum_i {n \brack i} x^{2i},\]
observing that the even Euler polynomials, apart from their leading monomials, contain only monomials that are odd powers of $x$.  Note that ${n \brack i}$ are integers and ${n \brack i}=0$ for $i\le 0$ or $i>n$. The coefficients of the even indexed Euler polynomials are listed in the online encyclopedia \cite{OEIS} as entry A060083. The first six lines in the triangle of ${n \brack i}$ are:
\begin{eqnarray*}
\begin{array}{ccccccccccc}
&&&&&1&&&&&\\
&&&&-1&&2&&&&\\
&&&3&&-5&&3&&&\\
&&-17&&28&&-14&&4&&\\
&155&&-255&&126&&-30&&5&\\
-2073&&3410&&-1683&&396&&-55&&6\quad,
\end{array}
\end{eqnarray*}
where the line and diagonal indexing starts with $n=1$ (read from top to bottom) and $i=1$ (read from left to right). We warn the reader that the recursions for the ${n \brack i}$ have no resemblance to those for the binomial coefficients.
\begin{theorem} \label{thm:oddexpansion}
Let $\varphi(x)=\sum_{i\ge 0}\gamma_i\,x^i$ be a formal power series.
Then $\varphi(x)$ is symplectic if and only if for each $n\ge 0$,
\begin{equation}\label{oddexpansion}
\gamma_{2n+1}=\sum_i {n \brack i} \gamma_{2i}.
\end{equation}
In particular, for
each choice of   $\gamma_0,\gamma_2,\gamma_4,\gamma_6,\dots$ there is a uniquely defined symplectic power series $\varphi(x)$ determined by the above rule.
\end{theorem}

\begin{corollary}
\label{Eulerrelations}
For all integers  $n,k,\ell$ we have
\begin{equation}\label{eq:Eulerrelations}
{n-k \brack \ell}+{n-\ell\brack k}={n \brack k+\ell} +\sum_{i}\sum_{r}{n \brack i}{r-1 \brack k}{i-r \brack \ell}.
\end{equation}
\end{corollary}

We would like to mention that the Euler coefficients are related to the Genocchi numbers $G_n$, respectively the Bernoulli numbers $B_n$, using the formula
\[{n\brack i}=-\frac{G_{2(n-i+1)}}{2(n-i+1)}{2n\choose 2i-1}=\frac{4^{n-i+1}-1}{n-i+1}B_{2(n-i+1)}{2n\choose 2i-1}.\]
This means that Equation \eqref{eq:Eulerrelations} can be interpreted as a cubic relation for the Bernoulli numbers.

\begin{proof}[Proof of Corollary \ref{Eulerrelations}] Let $\sum_{i\ge 0}\gamma_i\,x^i$ and $\sum_{j\ge 0}\delta_j\,x^j$ be symplectic power series. From Corollary \ref{cor:Product} we know that their Cauchy product
$\sum_{m\ge 0}\vartheta_m\,x^m$ is symplectic with, for each $m\in \N$, $\vartheta_m=\sum_{i+j=m}\gamma_i\:\delta_j$.
The left hand side of Equation \eqref{eq:Eulerrelations} arises from expressing
\[\vartheta_{2n+1}=\sum_{r+s=2n+1}\gamma_r\:\delta_s\]
in terms of even $\gamma$'s and $\delta$'s using Equation \eqref{oddexpansion}.
Similarly, the right hand side of Equation \eqref{eq:Eulerrelations} arises from expressing
\[\vartheta_{2n+1}=\sum_{i=1}^n{n\brack i}\vartheta_{2i}=\sum_{i=1}^n\sum_{r+s=2i}{n\brack i}\gamma_r\:\delta_s\]
in terms of even $\gamma$'s and $\delta$'s. In the argument, we also use the fact that
the even $\gamma$'s and $\delta$'s can be chosen freely.
\end{proof}

\begin{proof}[Proof of Theorem \ref{thm:oddexpansion}.] The argument is inspired by \cite[Section 7]{GesselUmbral}. There the situation is studied when two sequences  $(c_n)_{n\in\N}$ and $(d_n)_{n\in\N}$ are related by
\begin{eqnarray}
\label{eq:BasicBinomRel}
d_n=\sum_{i=0}^n{n\choose i}c_i
\end{eqnarray}
for all $n$.  Then \cite[Theorem 7.4]{GesselUmbral} states that for all nonnegative integers $m$ and $n$,
\begin{eqnarray}
\label{eq:Gesselmagic}
\sum_{i=0}^m{m\choose i}c_{n+i}=\sum_{j=0}^n(-1)^{n+j}{n\choose j}d_{m+j}.
\end{eqnarray}
By inspection of the generating function \eqref{Euler}, we derive the recursion
\[E_n(x)+\sum_{i=0}^n{n\choose i}E_i(x)=2x^n.\]
The idea is to put $c_i:=E_i(x)$ and $d_n:=2x^n-E_n(x)$ and observe that condition \eqref{eq:BasicBinomRel} holds. As the special case $n=m\ge 0$ of \eqref{eq:Gesselmagic}, we find
\begin{eqnarray}
\label{eq:EulSum}
\sum_{i=0}^n {n\choose i} E_{n+i}(x)=\sum_{i=0}^n (-1)^{n+i} {n\choose i}\left(2x^{n+i}-E_{n+i}(x)\right).
\end{eqnarray}
This can be rewritten as
\[\sum_{\substack{i=0\\n+i\text{ even}}}^n{n\choose i}\left(x^{n+i}-E_{n+i}(x)\right)=\sum_{\substack{i=0\\n+i\text{ odd}}}^n{n\choose i}x^{n+i},\]
which is equivalent to
\begin{eqnarray}\label{eq:doublestar}
\sum_{\substack{i=0\\n+i\text{ even}}}^n
(-1)^i{n\choose i}\sum_j {{n+i\over 2}\brack j}x^{2j-1}+
\sum_{\substack{i=0\\n+i\text{ odd}}}^n(-1)^i{n\choose i}x^{n+i}=0.
\end{eqnarray}

Let now $(\gamma_n)_{n\in \N}$ be a number sequence such that \eqref{oddexpansion} holds.
It will be enough to show that $\sum_{i \ge 0} \gamma_i\,x^i$ is symplectic. We interpret
$x$ as an umbral variable (cf. for example \cite{GesselUmbral}) and define the functional $\Gamma:\C[x]\to \C$ by
$\Gamma(x^n)=\gamma_{n+1}$ for $n$ odd. For even $n$, we put $\Gamma(x^n)=0$ (this choice
will not affect the considerations to follow). Applying $\Gamma$ to Equation \eqref{eq:doublestar},
we end up with
\begin{eqnarray*}
\sum_{\substack{i=0\\n+i\text{ even}}}^n
(-1)^i{n\choose i}\underbrace{\sum_j {{n+i\over 2}\brack j}\gamma_{2j}}_{=\gamma_{n+i+1}}+
\sum_{\substack{i=0\\n+i\text{ odd}}}^n (-1)^i{n\choose i}\gamma_{n+i+1}=0,
\end{eqnarray*}
showing that $\sum_{i=0}^n
(-1)^i{n\choose i}\gamma_{n+i+1}=0$, i.e., condition $(\operatorname S_{n+1})$ of Definition \ref{def:symplectic}.
\end{proof}

To complete this section, we use the above observations to indicate an alternate symplectic basis than those considered in Section \ref{sec:FormalComposite}.

\begin{lemma}
\label{lem:AltSympBasis}
For all integers  $k \geq 0$, the formal power series
\begin{eqnarray}\label{eq:AltSympBasis}
    \qquad\psi_k(x)
    :=
    \frac{1}{(2k-1)!}\sum\limits_{i=0}^\infty (-1)^{i-1} E_{i-1}^{(2k-1)}(0) \,x^i
    =
    - x^{2k} - \sum\limits_{i=k}^\infty {i \brack k}x^{2i+1}
\end{eqnarray}
is symplectic.
\end{lemma}
\begin{proof}
Using Equation \eqref{eq:EulSum}, we write
\begin{align*}
    \sum\limits_{i=0}^n {n \choose i} E_{n+i}(x)
    &={1\over 2}\sum_{i=0}^n {n \choose i} E_{n+i}(x)+{1\over 2}\sum_{i=0}^n(-1)^{n+i}{n \choose i}\left(2x^{n+i}-E_{n+i}(x)\right)\\
    &={1\over 2}\sum_{i=0}^n\left(1-(-1)^{n+i}\right){n \choose i} E_{n+i}(x)+\sum_{i=0}^n {n \choose i}(-x)^{n+i}\\
    &=\sum_{\substack{i=0\\n+i\text{ odd}}}^n{n \choose i}\underbrace{\left( E_{n+i}(x)-x^{n+i}\right)}_{=:\, (*)}+\sum_{\substack{i=0\\n+i\text{ even}}}^n{n \choose i}x^{n+i},
\end{align*}
where $(*)$ contains only even powers of $x$. Thus in $\sum_{i} {n \choose i} E_{n+i}(x)$, all coefficients of odd degree vanish, meaning that for all $k\ge 1$,
\[{1\over (2k-1)!}\sum_{i=0}^n{n\choose i} E_{n+i}^{(2k-1)}(0)=0.\]
It follows that $\sum\limits_{i=0}^\infty (-1)^{i-1} E_{n-1}^{(2k-1)}(0) \,x^i$ fulfills $(\mathrm S_{n+1})$.
\end{proof}

As a consequence of Lemma \ref{lem:AltSympBasis}, we have that $(\psi_n(x))_{n\in\N}$
forms a symplectic basis in the sense of Definition \ref{def:SympBasis}. Note that
$\psi_1$ is essentially the generating function of the Genocchi sequence \eqref{eq:phi1}, namely we have $\psi_1(x)=-\varphi_1(x)$. The idea of the above proof can be used to argue
that for each $\lambda\in\C$ the power series
\[\sum_{i\ge 0}(-1)^{i-1}\left(E_{i-1}(\lambda)-E_{i-1}(-\lambda)\right)\, x^i\in\C[\![x]\!]\]
is symplectic.


\section{Sample Calculations}
\label{sec:Calculations}

In this section, we survey a few special cases of Conjecture \ref{HigherRelations}
and Theorem \ref{TheoremTorus} that can be verified by direct computations.
We first consider the case of a unitary representation of a
finite group.  In this case, as a consequence of Corollary \ref{CorRatSympOrderD},
Conjecture \ref{HigherRelations} is a special case of Watanabe's Theorem
\cite{WatanabeGor1,WatanabeGor2}, see in particular
\cite[Theorem 7.1]{StanleyInvFinGp}.


\subsection{Quotients by finite unitary group representations}
\label{subsec:FiniteGroups}

Let $G$ be a finite group and $G\to \U(V)$ a unitary representation.
For $g \in G$, we let $g_V\co V\to V$ denote the corresponding
linear transformation.  Let $W:= V \times \cc{V}$, and then $G$ acts on $W$ via
$g_W\co (u, \overline{v}) \mapsto (g_V u, (g_V^{-1})^t \cc{v})$.
We identify $\R[V]$ with the subring $\C[W]^{-}$ of $\C[W]$
given by those elements fixed by complex conjugation, and then by Molien's formula
\cite{Molien}, see also \cite{DerskenKemperBook,SturmfelsBook}, the Hilbert series
of real regular invariants is given by
\begin{equation}
\label{eq:Molien}
    \Hilb_{\R[V]^G|\R}(t)
    =
    \frac{1}{|G|}\sum\limits_{g\in G} \frac{1}{\det(\id_W - g_W^{-1} t)}.
\end{equation}
Fix $g \in G$ and choose a basis for $V$ with respect to which $g_V$ is diagonal,
say $g_V = \diag(\lambda_1, \ldots, \lambda_n)$ where $|\lambda_i| = 1$ for each $i$.
Choosing the conjugate basis for $\cc{V}$ and concatenating to form a basis for $W$, we have
$g_W = \diag(\lambda_1,\ldots,\lambda_n,\cc{\lambda_1},\ldots,\cc{\lambda_n})$.
It then follows that
\begin{equation}
\label{eq:MolienTerm}
    \frac{1}{\det(\id_W - g_W^{-1} t)}
    =
    \prod\limits_{i=1}^n \frac{1}{(1 - \lambda_i t)(1 - \lambda_i^{-1} t)}.
\end{equation}

In this case, each term of the sum in
Equation \eqref{eq:Molien} is symplectic of order $2n$ at $t = 1$.
Specifically, for $\lambda \in \C$, define
\[
    f_\lambda(t) := \frac{1}{(1-\lambda t)(1 - \lambda^{-1} t)}.
\]
Then by the above observations, we have that each term in Molien's formula
is given by a product of $f_{\lambda_i}(t)$.  We claim the following.

\begin{lemma}
\label{lem:FiniteBaseCase}
For nonzero $\lambda\in\C$, the function $f_\lambda(t)$ is symplectic at
$t=1$ of order $2$.
\end{lemma}
\begin{proof}
If $\lambda = 1$, then $f_1(t) = (1 - t)^{-2}$ and the result is trivial,
so assume not.  Define
\[
    \rho_\lambda(y) = \frac{\lambda y}{1 + 2\lambda + \lambda^2 + \lambda y},
\]
and then by a simple computation,
\[
    \frac{1}{(1 - t)^2} \:\rho_\lambda\!\left(\frac{(1 - t)^2}{t}\right)
    =
    f_\lambda(t).
\]
The result then follows from Corollary \ref{cor:Product}.
\end{proof}

We remark that Lemma \ref{lem:FiniteBaseCase} can also be seen using the expansion
\[
    f_\lambda(t)
    =
    \sum\limits_{k=-2}^\infty
        \frac{-\lambda (\lambda^{k+1} - (-1)^{k+1})}{(\lambda^2-1)(\lambda-1)^{k+1}}
        (1 - t)^k
\]
and verifying \eqref{Sm} directly, or by checking that $f_\lambda(t)$ satisfies
\eqref{FuncEq t=a} for $a = 1$ and $d = 2$.

By Lemma \ref{lem:FiniteBaseCase} and Corollary \ref{cor:Product}, it follows that the expression
in Equation \eqref{eq:MolienTerm} is symplectic of order $2n$ at $t = 1$.


\subsection{Symplectic quotients by $\Sp^1$}
\label{subsec:CircleActions}

We observe that Corollary \ref{cor:Product} (in the case $a=1$) is a sufficient tool to verify Conjecture
\ref{HigherRelations} for symplectic $\Sp^1$-quotients case-by-case.  When the action is
\emph{generic}, i.e. no two weights have the same absolute value, an algorithm
for computing the Hilbert series is described in \cite[Section 4]{HerbigSeaton}
and has been implemented on \emph{Mathematica} \cite{Mathematica}.
To check that a concrete Hilbert series is symplectic of order $d=\dim(M_0)$, we use the
substitution
\[t\mapsto {1\over 2}\left(y+2\pm\sqrt{y(y+4)}\right)\]
as a heuristic to find a rational function $\rho(y)$
as in Corollary \ref{cor:Product}.

As an example, if $M_0$ is a reduced space associated to the weight vector
$(\pm 1, \pm 2, \pm 3)$ where not all weights have the same sign, then
\[
    \Hilb_{\R[M_0]}(t)
    =
    \frac{t^{10} + t^8 + 3t^7 + 4t^6 + 4t^5 + 4t^4 + 3t^3 + t^2 + 1}
    {(1 - t^2)(1 - t^3)(1 - t^4)(1 - t^5)}.
\]
One easily checks that
\[
    \rho(y)
    =   \frac{y^5 + 10y^4 + 36y^3 + 59 y^2 + 50 y + 22}
        {(y + 2)(y + 3)(y + 4)(y^2 + 5y + 5)}
\]
satisfies the condition of Corollary \ref{cor:Product}.

Similarly, let $\Sp^1$ act on $\C^n$ with weight vector $(\pm 1, \ldots, \pm 1)$ and let $M_0$
denote the symplectic quotient, which we note has dimension $2n-2$.
We assume that $n \geq 2$ and not all weights have  the same sign, for otherwise
$M_0$ is a point and the result is trivial.  By \cite[Section 5.3]{HerbigSeaton},
the Hilbert series $\R[M_0]$ is given by
\[
    \Hilb_{\R[M_0]}(t)
    =  (1 - t^2) \;{}_2 F_1(n,n,1,t^2)
    =   \frac{1}{(1 - t^2)^{2n-2}} \sum\limits_{k=0}^{n-1} {n-1 \choose k}^2 t^{2k}
\]
where ${}_2 F_1$ denotes the hypergeometric function; see \cite{BatemanTransFunBook}.
Theorem \ref{TheoremTorus} can be verified directly using Corollary \ref{cor:Product}
for this case as follows.

Define
\[
    \rho(y)
    =
    \frac{1}{(y+4)^{n-1}}\sum\limits_{k=0}^{n-1}
        {n-1\choose k}{2k\choose k} y^{n-k-1}.
\]
Let $(a)_b$ denote the Pochhammer symbol, i.e.
$(a)_b := a(a+1)\cdots (a+b-1)$ for $b>0$ and $(a)_0 = 1$,
and note that $(1/2)_k = (2k)!/(4^k k!)$.  Applying the definition of
${}_2 F_1$, we compute
\begin{align*}
    \frac{1}{(1 - t)^{2n-2}} \:\rho\!\left(\frac{(1 - t)^2}{t}\right)
    &=
    \frac{1}{(1 + t)^{2n-2}}\sum\limits_{k=0}^{n-1}
                {n-1\choose k}{2k \choose k} \left(\frac{t}{(1 - t)^2}\right)^k
    \\&=
    \frac{1}{(1 + t)^{2n-2}}\sum\limits_{k=0}^{n-1}
                \frac{(1-n)_k (1/2)_k}{ (1)_k  k!}
                \left( \frac{-4t}{(1 - t)^2} \right)^k
    \\&=        \frac{1}{(1 + t)^{2n-2}}\;
                {}_2 F_1(1-n,1/2,1,-4t/(1-t)^2).
\end{align*}
We then apply \cite[page 113, Equation (36)]{BatemanTransFunBook} and continue
\begin{align*}
    =           \frac{1}{(1 - t^2)^{2n-2}}\;
                {}_2 F_1(1-n,1-n,1,t^2)
    &=          \frac{1}{(1 - t^2)^{2n-2}}
                \sum\limits_{k=0}^{n-1} \frac{(1-n)_k (1-n)_k}{(1)_k k!}t^{2k}
    \\&=        \frac{1}{(1 - t^2)^{2n-2}}
                \sum\limits_{k=0}^{n-1} {n-1\choose k}^2 t^{2k}
    \\&=        \Hilb_{\R[M_0]}(t).
\end{align*}


\bibliographystyle{amsplain}
\bibliography{HerbigHerdenSeaton}

\end{document}